\documentclass[spanish, 11pt, a4paper, twoside,leqno]{article}
\usepackage{amsmath,amsfonts,amssymb,amscd,amsthm}
\usepackage{t1enc}
\usepackage[utf8]{inputenc}
\usepackage[english]{babel}
\usepackage[dvips]{graphicx}
\usepackage{pb-diagram}
\usepackage{color}
\usepackage{graphpap,array}
\usepackage{geometry}
\usepackage{fancyhdr}
\usepackage{hyperref}
\newif\ifcomments

\let\newComments\newKibitzer

 \newComments\EG{EG}{blue}
\newComments\LV{LV}{red}

\fancypagestyle{mypagestyle}{%
  \fancyhf{}
  \fancyhead[OC]{Esther Galina, Lorena Valencia}
 \fancyhead[EC]{Nilpotent orbits of Kac-Moody algebras and its parameterization for $\s$}
  \fancyfoot[C]{\thepage}%
}
\title{Nilpotent orbits of Kac-Moody algebras and its parameterization for $\s$}
\author{Esther Galina$^{1,2}$, Lorena Valencia$^2$}

\date{%
    $^1$Universidad Nacional de Córdoba, Facultad de Matemática, Astronomía, Física y Computación (FAMAF), Córdoba, Argentina\\%
    $^2$CIEM, Universidad Nacional de Córdoba, CONICET, Córdoba, Argentina\\%
    \today
}

\begin{document}

\newcommand{\G}{\mathcal G}
\newcommand{\Gd}{\dot{G}}
\newcommand{\Gm}{\mathcal G^{min}}
\newcommand{\oL}{\overline{\mathfrak L}}
\newcommand{\LGo}{\overline{\mathfrak{L}}(\dot{G})}
\newcommand{\LGh}{\widehat{\mathfrak{L}(\dot{G})}}
\newcommand{\LPSo}{\overline{\mathfrak{L}}(\mathrm{PSl}_n(\mathbb{C}))}
\newcommand{\LSo}{\overline{\mathfrak{L}}(\mathrm{Sl}_n(\mathbb{C}))}
\newcommand{\LS}{\mathfrak{L}(\mathrm{Sl}_n(\mathbb{C}))}
\newcommand{\g}{\mathfrak g}
\newcommand{\hg}{\hat{\mathfrak g}}
\newcommand{\gd}{\dot{\mathfrak{g}}}
\newcommand{\Lg}{\mathfrak{L}(\dot{\mathfrak{g}})}
\newcommand{\Lgbar}{\widetilde{\mathfrak{L}}(\dot{\mathfrak{g}})}
\newcommand{\Lggor}{\hat{\mathfrak{L}}(\dot{\mathfrak{g}})}
\newcommand{\gc}{\hat{\mathfrak g}}
\newcommand{\nm}{\mathbb{C}[[t]]\otimes\dot\eta^+}
\newcommand{\nk}{\mathcal K\otimes\dot\eta^+}
\newcommand{\h}{\mathfrak h}
\newcommand{\hd}{\dot{\mathfrak{h}}}
\newcommand{\ms}{\mathfrak s}
\newcommand{\A}{\mathcal A}
\newcommand{\B}{\mathcal B}
\newcommand{\bo}{\mathfrak b}
\newcommand{\C}{\mathbb C}
\newcommand{\K}{\mathcal K}
\newcommand{\m}{\mathfrak{m}}
\newcommand{\n}{\eta}
\newcommand{\nd}{\dot{\eta}}
\newcommand{\ngor}{\hat{\eta}^{+}}
\newcommand{\Or}{\mathcal O}
\newcommand{\W}{\mathcal W}
\newcommand{\Wd}{\dot{\mathcal W}}
\newcommand{\+}{\oplus}
\newcommand{\glK}{\mathfrak{gl}_n(\mathcal{K})}
\newcommand{\gln}{\mathfrak{gl}_n(\mathbb{C})}
\newcommand{\sln}{\mathfrak{sl}_n(\mathbb{C})}
\newcommand{\so}{\mathfrak{sl}_2(\mathbb{C})}
\newcommand{\s}{\mathfrak{sl}_n^{(1)}(\mathbb C)}
\newcommand{\sgor}{\hat{\mathfrak{sl}}_n^{(1)}(\mathbb C)}
\newcommand{\st}{\mathfrak{sl}_3^1}
\newcommand{\PSl}{\mathrm{Sl}_n(\mathbb{C})}
\newcommand{\Slk}{\mathrm{Sl}_n(\mathbb{C}[[t]])}
\newcommand{\slnk}{\mathfrak{sl}_n(\mathbb{C}[[t]])}
\newcommand{\Ad}{\mathrm{Ad \ }}
\newcommand{\ad}{\mathrm{ad}}
\newcommand{\adg}{ad_{\mathfrak g}}
\newcommand{\adt}{\widetilde{\mathrm{ad}}}
\newcommand{\adc}{ad_{\mathfrak{g}_{comp}}}
\newcommand{\Exp}{\mathrm{Exp}}
\newcommand{\Aut}{\mathrm{Aut}}
\newcommand{\End}{\mathrm{End}}
\newcommand{\TS}{\{ H,X,Y \} }
\newtheorem{theorem}{\textbf{Theorem}}[section]
\newtheorem{prop}[theorem]{\textbf{Proposition}}
\newtheorem{cor}[theorem]{\textbf{Corollary}}
\newtheorem{lemma}[theorem]{\textbf{Lemma}}
\newtheorem{defn}{\textbf{Definition}}[section]
\newtheorem{obs}{\textit{Remark}}[section]
\newtheorem{exmp}{\textbf{Example}}[section]

\maketitle

\begin{abstract}
In the context of affine complex Kac-Moody algebras, we define the meaning of nilpotent orbits under the adjoint action of the maximal Kac-Moody group. We also give a parameterization of nilpotent orbits of $\s$.
\end{abstract}

\section{Introduction} 

Let $\gd$ be a finite dimensional complex semisimple Lie algebra. Consider the adjoint group
$\dot{G}_{\ad}$ of $\gd$ and the adjoint representation $\Ad$  of $\dot{G}_{\ad}$ over $\gd$. We say $x\in\gd$
is nilpotent in $\gd$ if $\ad x$ is a nilpotent endomorphism over $\gd$. If $x\in\gd$ is  nilpotent  then so is $\Ad g (x)$ for all $g\in\dot{G}_{\ad}$. This fact allows us to define nilpotent orbits in a semisimple Lie algebra: if $x\in
\dot{G}_{\ad}$ is nilpotent, then the \textit{nilpotent orbit} through $x$ is
$$\mathfrak{O}_x=\Ad \dot{G}_{\ad} (x).$$

Nilpotent orbits have several applications in representation  theory of Lie algebras and groups, algebraic groups, Weyl groups, and other related objects. For that, there has been huge interest in their study and classification. Dynkin and Kostant gave a classification of such orbits by using weighted Dynkin diagrams [Co]. In the particular case $\gd=\mathfrak{sl}_n(\C)$, as it is a semisiple Lie algebra over an algebraically closed field, the nilpotent elements in $\gd$ are the nilpotent endomorphisms. In general, for the classical complex semisimple  Lie algebras, the normal Jordan form allows us to parameterize the nilpotent orbits of $\gd$ by partitions of $n$. In the case of  exceptional complex semisimple  Lie algebras, even though there is a correspondence  between nilpotent orbits and weighted Dinkyn diagrams, Bala and Carter determined which weighted Dynkin diagram corresponds to a nilpotent orbit of the algebra. For classical and exceptional real semisimple Lie algebras the classification was done by Djokovic in \cite{Dj2} and \cite{Dj1}. Later, Nöel gives an analogous of the Bala-Carter classification for the real case \cite{N} and there is a way to associate a weighted Vogan diagram to each real nilpotent orbit in \cite{Ga}.

The complex Kac-Moody algebras are generalizations of finite-dimensional complex semisimple Lie algebras defined by generators and relations from a generalized Cartan matrix.
There are three types of indecomposable Kac-Moody algebras:  those of finite type (finite-dimensional semisimple Lie algebras), the affine and the indefinite ones. The affine Kac-Moody algebras are the most widely studied  infinite dimensional Lie algebras and they have several applications. They are completely classified and they have a very interesting representations theory. Both, the classification and its representation theory, are similar to those of simple Lie algebras. 

This analogy has motivated the question of whether it is possible to speak of nilpotent orbits in the affine context and, in this case, how to obtain a classification or parameterization of these objects.
The results obtained, presented in this paper, are the product of the doctoral thesis of Lorena Valencia at the Universidad Nacional de Córdoba (Argentina).

As a first step,  it is necessary to choose an appropriate group acting on the affine Kac-Moody algebra that assumes the role of the adjoint Lie group in the case of simple Lie algebras of finite dimension. 

Let's $\gd$ a complex simple Lie algebra and $\mathfrak{L}(\dot\g):=\C[t,t^{-1}]\otimes_{\C}\gd$ the loop algebra associated with $\gd$. Then, the algebra
$$\g=\mathfrak{L}(\dot\g)\oplus\C c\+\C d$$ 
is an  \emph{affine Kac Moody algebra} with bracket given by
\begin{align*}
[t^m\otimes x+\lambda c+\mu d,&t^n\otimes y+\lambda'c+\mu'd]= \\
    &=t^{m+n}\otimes[x,y] +\mu n t^n\otimes y-\mu'mt^m\otimes x+m\delta_{m,-n}\langle x,y\rangle c
\end{align*}
where $\langle,\rangle$ denotes the Killing form over $\gd$ and
$\lambda,\mu,\lambda',\mu'\in\C$, $m,n\in\mathbb{Z}$,
$x,y\in\dot\g$.

Define the extension of $\g$ as the algebra
$$\hg:=\C[[t]][t^{-1}]\otimes_{\C}\gd\+\C c\+\C d$$ 
whose bracket  naturally extends the bracket of $\g$.

The correspondence between finite-dimensional complex semisimple Lie algebras and connected and simply connected Lie groups has been extended by Kac-Peterson to the correspondence between complex Kac-moody algebras and certain simply connected topological groups known as Kac Moody groups \cite{Ku}.

In the literature, there are several constructions of Kac-Moody groups associated with a Kac-Moody algebra. We will work with the maximal Kac-Moody group \cite{Ku}. It has different presentations. In the more general case, it is the amalgamated product of a certain system of groups. We will adopt a special presentation for Kac-Moody groups associated with affine Kac-Moody algebras based on the loop group of the finite-dimensional Lie algebra.

In general, a maximal Kac-Moody groups $\G$ is an ind-algebraic group and shares several important properties with simple algebraic groups, as the Bruhat decomposition and the generalized Gaussian decomposition.

Denote by $\Ad$ the adjoint representation of $\G$ over the algebra $\hg$.  We say that  $X\in\g$ (resp. $X\in\hg$) is \emph{locally nilpotent} over $\g$ (resp. over $\hg$), if the endomorphism 
$\ad_{\g} X$   is locally nilpotent over $\g$ (resp. $\ad_{\hg}$ is locally nilpotent over $\hg$). 
All the locally nilpotent elements of an extended affine Kac-Moody algebra are actually nilpotent. 
Moreover, for all nilpotent element  $X\in\hg$ and $g\in\G$, $\Ad g(X)$ is also nilpotent. Therefore, nilpotent orbits in $\hg$ make total sense. 

In this work, we present some properties of nilpotent elements of affine Kac-Moody algebras and we give a parametrization of the nilpotent orbits of  $\g=\s$.

The paper is organized as follows. In section 2 we defined Kac-Moody algebras and groups as extensions of loop algebras and loop groups respectively. Section 3 presents the central objects of this work, that is nilpotent elements and nilpotent orbits of affine Kac-Moody algebras under the adjoint action of the group $\G$ over the algebra $\hg$, and some structure results of these orbits. In particular, we prove every nilpotent element in $\g$ is conjugated by $\G$ to an element of $\nm\+\C c$, moreover, every nilpotent element is $\G$-conjugated  to one whose component in $\mathfrak{sl}_n(\C[[t]][t^{-1}])$ is an upper triangular matrix with entries in $\C[[t]]$.
Section 4 focuses on nilpotent orbits of the affine Kac-Moody algebra $\s$. We introduce a set of elements in $\mathfrak{sl}_n(\C[[t]][t^{-1}])$ denominated cuasi-Jordan matrix, whose entries are all zero except at the entries  $(i,i+1)$. We obtain that every nilpotent orbit contains an element of type $D+\lambda c$, where $D$ is a cuasi-jordan matrix and $\lambda\in\C$. These distinguished elements in each nilpotent orbit of $\s$ permits us to obtain the principal result of this paper, the classification of the nilpotent orbits in $\s$, given by the following theorem:

\textbf{Theorem:}  There is a biyective correspondence between nilpotent orbits in $\s$ and the set 
$\{(\sigma,j):\sigma=[i_1,...,i_d]\in\mathcal{P}(n) \ \mathrm{ y } \ 0\leq j<i_d\}\times\C$. 

\section{Preliminaries}

\subsection{Kac Moody algebras} 
Let $\A=\C[t,t^{-1}]$ be the algebra of Laurent polynomials and
$\mathcal{K}=\C[[t]][t^{-1}]$ the field of Laurent polynomial series.
Consider $\dot\g$ a finite dimensional complex semisimple Lie algebra  and $\mathfrak{L}(\dot\g):=\A\otimes_{\C}\dot\g$ the loop algebra associated with $\gd$, with the Lie algebra structure given by
\begin{equation}\label{loop}
    [P\otimes x,Q\otimes y]=PQ\otimes[x,y]
\end{equation}
for any $P,Q\in\A$, $x,y\in\dot\g$.

Denote by $\widetilde{\mathfrak{L}}(\dot\g)$ the central extension of Lie algebra $\mathfrak{L}(\dot\g)$ associated with the  2-cocycle $\nu$ defined over $\mathfrak{L}(\dot\g)$ by
$$\nu(P\otimes x,Q\otimes y)=res(\frac{dP}{dt}Q)\kappa(x,y)$$
where $P,Q\in\A$, $x,y\in\gd$ and $\kappa$ is the Killing form of $\gd$.

Then,
$\widetilde{\mathfrak{L}}(\dot\g):=\mathfrak{L}(\dot\g)\oplus\C c$ and its bracket is given by 
\begin{gather}
[f,c]=0  \\
    [f,g]=[f,g]_{\mathfrak{L}(\dot\g)}+\nu(f,g)c
\end{gather}
for all $f,g\in\mathfrak{L}(\gd)$, where
$[f,g]_{\mathfrak{L}(\dot\g)}$ is the bracket of
$\mathfrak{L}(\dot\g)$.

The  \emph{(untwisted) affine Kac Moody  algebra }associated with $\dot\g$ is the algebra
\begin{center}
$\g=\mathfrak{L}(\dot\g)\oplus\C c\+\C
d=\widetilde{\mathfrak{L}}(\dot\g)\+\C d.$
\end{center}
It is the result of adding a derivation $d$ to 
$\widetilde{\mathfrak{L}}(\dot\g)$, where $d$ acts on 
$\mathfrak{L}(\dot\g)$ by $t\frac{d}{dt}$ and annihilates $c$. The bracket is given by
\begin{multline}\label{corchete}
    [t^m\otimes x+\lambda c+\mu d,t^n\otimes y+\lambda'c+\mu'd]=\\ 
    =t^{m+n}\otimes[x,y] +\mu n t^n\otimes
y-\mu'mt^m\otimes x+m\delta_{m,-n}\langle x,y\rangle c
\end{multline}
for $\lambda,\mu,\lambda',\mu'\in\C$, $m,n\in\mathbb{Z}$ and
$x,y\in\dot\g$.

The application $x\mapsto 1\otimes x$ is an embedding of Lie algebras, 
$\gd\hookrightarrow\g$. So, we can consider the algebra 
$\gd$ as a subalgebra of $\g$. If $\dot{\h}$ is a 
Cartan subalgebra of $\gd$, then $\h=\dot{\h}\oplus\C
c\oplus\C d$ is a Cartan subalgebra of $\g$. Moreover, if $\{e_i,f_i|i=1,\cdots,n\}$ are the Chevalley's generators of $\gd$, then $\{E_i,F_i|i=0,\cdots,n\}$ are the Chevalley's generators of $\g$, where $E_i=1\otimes e_i$, $F_i=1\otimes f_i$ for $i=1,\cdots,n$ and $E_0, F_0$ are obtained as follows: consider $\omega$ the linear involution of $\dot\g$, defined by
\begin{center}
$\omega(e_i)=-f_i, \ \omega(f_i)=-e_i, \ \omega(h_i)=-h_i,$
\end{center}
choose $f_0\in\dot\g_{\theta}$, where $\theta\in\dot{\Delta}$ is the  root of maximal weight, such that 
\begin{center}
$\langle f_0,\omega(f_0)\rangle=-1$
\end{center}
and define $e_0=-\omega(f_0)\in\dot\g_{-\theta}$. Therefore,
\begin{center}
$E_0=t\otimes e_0, \ F_0=t^{-1}\otimes f_0$
\end{center}

Let's $\delta\in\h^*$ given by $$\delta|_{\hd\+\C c}\equiv 0, \ \delta(d)=1.$$
Set $\Pi=\{\alpha_0=\delta-\theta\,\alpha_1,\cdots,\alpha_n\}$ the simple root system of $\g$, where $\dot\Pi=\{\alpha_1,\cdots,\alpha_n\}$ is the simple root system of $\gd$.

Define the completeness of 
 $\g$ as the algebra
\begin{equation}\label{hg}
\hat{\g}:=\K\otimes_{\C}\gd\+\C c\+\C d
\end{equation}
with the linear extension of the bracket given by (\ref{corchete}). As a finite number of negative powers of $t$ are allowed, this bracket has sense.

The set of roots of $(\g, \h)$  is  $$\Delta=\{j\delta;j\in\mathbb Z
\smallsetminus\{0\}\}\cup\{j\delta+\beta;j\in\mathbb
Z,\beta\in\dot\Delta\},$$  where $\dot\Delta$ is the root system of $(\g, \h)$.

\

\subsection{Weyl group} \ Let $\Wd$ be the Weyl group of $\gd$, and 
$\dot{\mathcal{Q}}^{\vee}:=\oplus_{i=1}^l\mathbb{Z}\alpha_i^{\vee}\subset\hd$
the reticulum of coroots of $\gd$. This means that $\Wd$ has a canonical action over
 $\dot{\mathcal{Q}}^{\vee}$  obtained from its action on $\hd$. The affine Weyl group is the semidirect product 
$\mathrm{Aff} \ \Wd:=\Wd\ltimes\dot{\mathcal{Q}}^{\vee}$, for $q\in\dot{\mathcal{Q}}^{\vee}$ and 
$\tau_q=(1,q)\in\mathrm{Aff} \ \Wd$.

If $\W$ is the Weyl group associated with the affine Kac Moody algebra $\g$, there exits a (unique) group isomorphism 
\begin{equation}
\beta:\W\rightarrow\mathrm{Aff} \ \Wd
\end{equation}
such that $s_0\mapsto\tau_{\theta^{\vee}}\gamma_{\theta}$ and
$s_i\mapsto\dot s_i$, for $1\leq i\leq n$, where $\{s_0,...,s_n\}$
are the simple reflections of $\W$, $\{\dot s_1,...,\dot s_n\}$ the
simple reflections of $\Wd$ and $\gamma_{\theta}\in\Wd$ is the 
corresponding one to the highest root $\theta$ ([Ku], Proposition 13.1.7).

\

\subsection{Kac Moody groups} \ As well as affine Kac Moody algebras can be presented as central extensions of loop algebras, affine Kac Moody groups can be seen as central extensions of loop groups.

Let $\Gd$ be a connected simple complex algebraic group
with Lie  algebra $\dot{\mathfrak g}$.   The set $\Gd(\mathcal K)$ of  $\mathcal K$-rational points of
the algebraic group $\Gd$ is the loop group associated with $\Gd$.

By hypothesis $\Gd$ is a subgroup of $SL_N(\mathbb{C})$, for a convenient $N$. Let $I$ be the ideal of $\mathbb{C}[SL_N(\mathbb{C})]$ such that $\C[\Gd]=\mathbb{C}[SL_N(\mathbb{C})]/ I$. 
Then, under the canonical identification,
\begin{center}
    $\Gd(\mathcal K)=\{g=(g_{ij})\in SL_N(\K):P(g_{ij})=0, \text{ for all } P\in
I\}$
\end{center}

Extend $\Gd(\mathcal K)$ by adding $''\mathrm{Exp} \ d''$ as follows. Denote the group of automorphisms of $\C$-algebras of $\mathcal K$ by $\mathrm{Aut}(\mathcal K)$. Let
$\gamma:\C^*\rightarrow\mathrm{Aut} \ \mathcal K$
be the homomorphism of groups defined as $\gamma(z)(P(t))=P(zt)$,
for $z\in\C^*$ and $P\in\mathcal K$. It canonically induces a group  
homomorphism
$\gamma_{\Gd}:\mathbb{C}^*\rightarrow\mathrm{Aut} \ \Gd(\mathcal K)$
given by $\gamma_{\Gd}(z)(P_{ij}(t))=(P_{ij}(zt))$, where $\mathrm{Aut} \ \Gd(\mathcal K)$ is the group of automorphisms of  $\Gd(\mathcal K)$.

Define the semidirect product group as 
\begin{equation}\label{LGo}
    \LGo:=\mathbb{C}^*\ltimes \Gd(\mathcal K)
\end{equation}

For all $z\in\mathbb{C}^*$, denote by $d_z$ the corresponding element
$(z,1)\in\LGo$.

The \emph{Kac Moody group} $\G$ associated with the affine Kac Moody algebra $\g$ is the central extension of the group $\LGo$.

\

\subsection{Adjoint representation} The adjoint representation of  $\Gd$ on $\gd$ extends to 
a representation $\mathrm{Ad}_{\mathcal K}$ of $\Gd(\mathcal K)$ on
$\mathcal K\otimes_{\C}\gd$. By the election of the inclusion of algebraic groups 
$\Gd\hookrightarrow\mathrm{SL}_N(\C)$, we have that
\begin{equation}\label{ad}
    \mathrm{Ad}_{\mathcal K} g(x)=gXg^{-1}
\end{equation}
for $g\in\Gd(\mathcal K)\subset\mathrm{SL}_N(\mathcal K)$ and
$X\in\mathcal K\otimes_{\C}\gd\subset\mathrm{M}_N(\mathcal K)$ (where
$\mathrm{M}_N(\mathcal K)$ is the space of  $N\times N$-matrix over $\mathcal K$).

In this case, the adjoint representation $\mathrm{Ad}$ of $\LGo$ in
$\hg$ is calculated as follows: for $g\in \Gd(\mathcal K)$,
$x\in \mathcal K\otimes_{\mathbb{C}}\gd$, $\lambda, \mu\in\mathbb{C}$ and
$z\in\mathbb{C}^*$
\begin{equation}\label{adjointLG}
\begin{split}
\Ad g (x+\lambda c+\mu d) := & \mathrm{Ad}_{\mathcal K} g (x)-\mu
t(\frac{dg}{dt})g^{-1}+ \\
   &  + \left(\lambda- \mathrm{res}\left\langle
g^{-1}\frac{dg}{dt},x-\frac{1}{2}\mu
tg^{-1}\frac{dg}{dt}\right\rangle_{t}\right)c+\mu d \\
  \Ad d_z (x+\lambda c+\mu d) := &\gamma_{\hg}(z)(x)+\lambda
  c+\mu d
  \end{split}
\end{equation}
where $\langle , \rangle_t$ is the  $\mathcal K$-bilineal form over
$\mathcal K\otimes_{\C}\gd$ that extends the invariant bilinear form 
normalized by: $\langle p\otimes x,q\otimes y\rangle=pq\langle x,y\rangle$ where $\langle,\rangle$ 
is the Killing form over $\gd$, $\mathrm{res}$ denotes the 
coefficient of $t^{-1}$,
$\gamma_{\gc}:\C^*\rightarrow\mathrm{Aut} \ \gc$ is the induced application
from  $\gamma$ (similar to the application $\gamma_{\Gd}$),
$g=(g_{ij})_{1\leq i,j\leq N}\in\mathrm{SL}_n(\mathcal K)$,
$\frac{dg}{dt}$ is
$(\frac{dg_{ij}}{dt})_{ij}\in\mathrm{M}_n(\mathcal K)$ and
$(\frac{dg}{dt})g^{-1}=g^{-1}\frac{dg}{dt}\in\mathcal K\otimes_{\C}\gd$.

\

\subsection{Subalgebra $\mathfrak{b}_w$} \ For all $w\in\W$, let
\begin{equation}\label{Deltaw}
\Delta_w=\{\alpha\in\Delta^+|w^{-1}\alpha\in\Delta^-\}=\Delta^+\cap
w\Delta^- .
\end{equation}
By  ([Ku], Lemma 1.3.14), if $w=s_{i_1}...s_{i_l}$ is a reduced expression we have that $\Delta_w=\{\alpha_{i_1},s_{i_1}\alpha_{i_2},...,s_{i_1}...s_{i_{l-1}}\alpha_{i_l}\}$.  Moreover, $\Delta_w$ is closed for the bracket, then $\bo_w:=\h\oplus\bigoplus_{\alpha\in\Delta_w}\g_{\alpha}$ is a finite dimensional subalgebra of $\g$.

\begin{defn}\label{pollo} Let $\mathfrak s$ be a Lie algebra, then $\mathfrak s$ is called
\begin{itemize}
  \item \emph{$\ad_{\mathfrak s}$-triangular}, if there exits a flag of ideals: $\mathfrak s_0=\{0\}\subset\mathfrak s_1\subset\mathfrak s_2\subset...$ such that
$\bigcup\mathfrak s_i=\mathfrak s$ and, for each $i\geq 0$,
$\mathrm{dim}(\mathfrak s_{i+1}/\mathfrak s_i)=1$
  \item \emph{$\ad_{\mathfrak s}$-locally diagonalizable}, resp. \emph{$\ad_{\g}$-finitly semisimple}, if $\mathfrak s$ is a direct sum of irreducible 1-dimensional $\mathfrak s$-submodules, resp. finitly-dimensional.
  \item \emph{$\ad_{\mathfrak s}$-locallly finite}, if all $x\in\mathfrak s$ is contained in an ideal of finite dimension of $\mathfrak s$.
\end{itemize}
\end{defn}

\begin{theorem}\label{co} Let $\mathfrak s\subset\g$ be a subalgebra, then the following conditions are equivalent,
\begin{enumerate}
\item $\mathfrak s$ is $\ad_{\g}$-triangular.
\item $\mathfrak s$ is a finitly dimensional Lie algebra  which is $\ad_{\g}$-locally finite.
\item There exits $g\in\G$ and $w\in\W$ such that $\Ad g (\mathfrak s)\subset\bo_w$.
\end{enumerate}
\end{theorem}
\begin{proof} Ver ([Ku] 10.2.5).
\end{proof}

\begin{obs}\label{pez} Every element $X\in\g$, $\ad_{\g}$-locally nilpotent is 
$\ad_{\g}$-locally finite. Moreover, the subalgebra of $\g$ generated by such an element is also 
$\ad_{\g}$-locally finite. Then, by Theorem \ref{co},
we have that for each  locally nilpotent element in  $\g$,
there exist $g\in\G$ and $w\in\W$ such that $\Ad g (X)\in\bo_w$ for some $w\in\W$.
\end{obs}

\

\subsection{Locally nilpotent elements} \ Let $X\in\g$ (resp. $X\in\hg$). We say that $X$
is \emph{locally nilpotent} on $\g$  (resp. on $\hg$), if
$\ad_{\g} X\in\End(\g)$ (resp.  $\ad_{\hg} X\in\End(\hg)$)  is
locally nilpotent.

\begin{prop} Let $X\in\hg$ $\ad_{\hg}$-locally nilpotent. Then,
$\Ad g( X)$ is $\ad_{\hg}$-locally nilpotent,  for all $g\in\G$.
\end{prop}
\begin{proof} Let $g\in\G$ and $X\in\hg$ be $\ad_{\hg}$-locally nilpotent. 
We have that $\Ad g\in\mathrm{Aut}(\hg)$, then for $Y\in\hg$
$$\begin{aligned}
\mathrm{ad}(\Ad g (X))(Y)&=[\Ad g(X),Y] \\
&=\Ad g([X,(\Ad g)^{-1}(Y)])  \nonumber  \\
&=(\Ad g)(\mathrm{ad}X)(\Ad g)^{-1}(Y)
\end{aligned}.$$
As $\ad X$ is locally nilpotent on $\hg$ and $\Ad g\in\Aut(\hg)$,
then $(\Ad g)(\mathrm{ad}X)(\Ad g)^{-1}$ is
locally nilpotent on $\hg$, which means that $\Ad g( X)$ also is.
\end{proof}

\

The previous proposition allows us to introduce the definition of
locally nilpotent orbits on affine Kac Moody algebras.

\begin{defn} Let $X\in\hg$. We say that the $\G$-\'orbit through $X$,
\begin{equation}
\mathfrak{O}_X=\Ad \G (X)=\Ad \LGo (X)
\end{equation}
is a \emph{locally nilpotent orbit on $\hg$} if each of its elements is
 $\ad_{\hg}$-locally nilpotent.
\end{defn}

Note that  locally nilpotent  orbits are seen in  $\hg$  and not in $\g$, because  $\Ad \G (\g)\subset\hg$. 
Denote by $\mathfrak O(\hg)$ the subset of orbits on
$\hg$, whose intersection with $\g$ is no empty.
Our interest is to characterize and classify the
orbits in $\mathfrak{O}(\hg)$. From now on, when we talk about a locally nilpotent orbit
on $\hg$, we mean an element of 
$\mathfrak O(\hg)$.

Until now, we have distinguished between the adjoint representation on 
 $\g$ and on $\hg$ by $\ad_{\g}$ and
$\ad_{\hg}$; from now on, we denote the adjoint on 
$\hg$ by $\ad=\ad_{\hg}$. We will say that an element in $\hg$ is locally nilpotent if it is $\ad_{\hg}$-locally nilpotent.

\section{Distinguished representatives}
In this section we will prove that every locally nilpotent orbit has an element in 
$\K\otimes_{\C}\nd^+\oplus\C c$, where $\nd^+=\bigoplus_{\alpha\in\dot{\Delta}^+}\gd_{\alpha}$ and $\dot{\Delta}^+$ 
are the positive roots of $\dot{\Delta}$. Moreover, we will show that every locally nilpotent element in $\g$ is nilpotent. 
\begin{lemma}\label{sup}
  For each locally nilpotent $X\in\g$,  there exists $g\in\G$ such that $\Ad g(X)\in\h\+\K \otimes_{\C}\nd^+$.
\end{lemma}
\begin{proof} Let $X\in\g$ be an $\ad_{\g}$-locally nilpotent element. As every $\ad_{\g}$-locally nilpotent element is $\ad_{\g}$-locally finite, by Theorem \ref{co}, there exist $g_1\in\G$ and $w\in\W$ such that $\Ad g_1 (X)=Y\in\bo_{w}$.

Consider $\Lambda=\{\alpha\in\dot\Delta:\alpha+n\delta\in\Delta_w$ for some $n\in\mathbb{Z}\}$. If $\alpha\in\Lambda$, then $-\alpha\notin\Lambda$; otherwise there would be $m,n\in\mathbb{Z}$ such that $\alpha+m\delta,-\alpha+n\delta\in\Delta_w$, which
implies that  $(m+n)\delta\in\Delta_w$ because $\Delta_w$ is closed, what is an absurd. In addition, as $\Lambda\subset\dot\Delta$ is closed, there exits $v\in 
\Wd$ such that $v\cdot\Lambda\subset\dot\Delta^+$. As $\Wd$ is generated by simple reflections $s_1,\cdots,s_n$ and $s_i(\alpha+m\delta)=s_i(\alpha)+m\delta$ for each $i=1,\cdots,n$, then  $v\cdot\Delta_w\subset\dot\Delta^+\+\mathbb{Z}_{\geq 0}\delta$. 

Now, for each $s_i\in\W$ exists $\tilde{s_i}\in\G$ such that $\Ad\tilde{s_i} (Y)=s_i (\ad)(Y)$. Moreover, $\Ad \tilde{s_i}(Y_{\alpha})\in\g_{s_i(\alpha)}$, for all $\alpha\in\Delta$ and $Y_{\alpha}\in\g_{\alpha}$. 
Then, if $v=s_{i_1}\cdots s_{i_k}$ is a reduced expression in $\W$,  $g_v=\tilde{s}_{i_1}\cdots \tilde{s}_{i_k}\in\G$ is such that $\Ad g_v(Y_{\alpha})\in\g_{v\alpha}$ for all $\alpha\in\Delta$ e $Y_{\alpha}\in\g_{\alpha}$. 

Therefore, if $g=g_v g_1$,  $\Ad g (X)\in\h\+\sum_{\lambda\in\Theta}\g_{\lambda}$, where $\Theta\subset\{\alpha+n\delta:\alpha\in\dot\Delta^+, n\in\mathbb{Z}_{\geq 0}\}$.
\end{proof}

\begin{lemma}\label{der} Let $Y=X+\lambda c+\mu d\in\hg$. If $Y$ is locally nilpotent, then $\mu=0$.
\end{lemma}
\begin{proof} Let $Y=X+\lambda c+\mu d\in\g$ be a locally nilpotent element. By Lemma \ref{sup}, assume that   $X\in\h\+\mathbb{C}[[t]]\otimes_{\C}\dot\eta^+$. For $k\in\mathbb N$ and $H\in\hd$,
\begin{align*}
\mathrm{ad}^k(X+\mu d+\lambda c)(& t\otimes H ) = (\mathrm{ad}X+\mathrm{ad}\mu d)^k (t\otimes H)\\
   & =\sum_{i_1+...+i_s=k}\mathrm{ad}^{i_1}(X)\mathrm{ad}^{i_2}(\mu d)...\mathrm{ad}^{i_{s-1}}(X)\mathrm{ad}^{i_s}(\mu d)(t\otimes H).
\end{align*}
Then, $\mathrm{ad}^k(X+\mu d+\lambda c)(t\otimes H)=X'+\mu^k t\otimes H$, where $X'\in\mathbb{C}[[t]]
\otimes_{\C}\dot\eta^+$, because each term in the sum belongs to $\mathbb{C}[[t]]\otimes_{\C}\dot\eta^+$, except the term  $\mathrm{ad}^k(\mu d)(t\otimes H)=\mu^k t\otimes H$.  Therefore, $\mathrm{ad}(X+\lambda c+\mu d)^k(t\otimes H)\neq 0$ for all $k\in\mathbb N$. But this is possible only if $\mu=0$.

Consider now any other element in the orbit of $Y$, it is of form $\mathrm{Ad}_{\K} g (X)+\tau c$ or $\gamma_{\hg}(z)(X)+\lambda c$ for some $g\in\G$ and $z,\tau\in\mathbb C$. Then, the component of the derivation of any nilpotent element is null.
\end{proof}

\begin{theorem}\label{triang} Let $X\in\g$ be a locally nilpotent element. Then, there exists $g\in\G$ such that $\Ad g (X)\in\nm\+\C c$
\end{theorem}
\begin{proof} Let $Y=H+X+\lambda c$ be a locally nilpotent element with $X\in\nm$, $H=\sum_{i=1}^n k_i \alpha_i^{\vee}$ and $\lambda\in\C$. If $H$ is non zero, then $k_j\neq 0$ for some $j\in\{1,\cdots,n\}$. Let's see that for all $k\in\mathbb N$,
$\ad^k(Y)(E_j)\neq 0$. 

For each $\alpha=\sum_{i=1}^n m_i\alpha_i \in\dot\Delta^+$ denote by $|\alpha|=\sum_{i=1}^n m_i$ and consider the sets 
$$\g_1=\bigoplus_{\substack{{\alpha+l\delta\in\Delta}\\{|\alpha|=1}}}\g_{\alpha+l\delta}, \ \ \ \ \ \g_2= \bigoplus_{\substack{{\alpha+l\delta\in\Delta}\\{|\alpha|\geq2}}}\g_{\alpha+l\delta} .$$
We have that $\ad^k(Y)(E_j)=r^k E_j+X_k,$ where $r=\sum_{i=1}^n k_i\alpha_j(\alpha_i^{\vee})\neq0$.
Moreover,  $r^k E_j\in\g_1$ and $X_k\in\g_2$ or $X_k=0$, which contradicts the fact that $Y$ is locally nilpotent. Hence, $H=0$ and the result of the theorem follows from Lemmas \ref{sup} and \ref{der}.
\end{proof} 

\begin{theorem} Every locally nilpotent element in $\g$ is nilpotent.
\end{theorem}
\begin{proof} If $Y+\lambda c\in\g$ is locally nilpotent in $\g$, with $Y=\sum_{\alpha\in\Lambda}t^{n_{\alpha}}\otimes Y_{\alpha}$ for some $\Lambda\subset\dot\Delta\cup\{0\}$, then $\overline{Y}=\sum_{\substack{\alpha\in\Lambda}}Y_{\alpha}\in\gd$ is nilpotent in $\gd$. Therefore, there exists $k\in\mathbb{N}$ such that $\mathrm{ad}^k(\overline{Y})|_{\gd}\equiv0$. 

Let $m\in\mathbb N$, $\beta\in\dot\Delta\cup\{0\}$ and $Z_{\beta}\in\gd$, we have that 
\begin{equation}
[Y,t^m\otimes Z_{\beta}]=\sum_{\alpha\in\Lambda}t^{n_{\alpha}+m}\otimes[Y_{\alpha},Z_{\beta}]+n_{\alpha}\delta_{n_{\alpha},-m}\kappa(Y,Z_{\beta})c.
\end{equation}
Hence, each term of $\mathrm{ad}^k(Y)(t^m\otimes Z_{\beta})$ is in the center of $\g$ or is of form $t^s\otimes W$ for some $s\in\mathbb{Z}$ and $W=[Y_{\alpha_{i_1}},[...[Y_{\alpha_{i_l}},Z_{\beta}]...]\in\gd$ for some $\alpha_{i_1},...\alpha_{i_l}\in\Lambda$. The nilpotency of $\overline{Y}$ in $\gd$ implies that $W=0$, hence $\mathrm{ad}^{k+1}(Y)(t^m\otimes Z_{\beta})=0$. The same argumet works to show that 
$\mathrm{ad}^{k+1}(Y)(d)=0$. Therefore, $Y$ is nilpotent in $\g$. Moreover, it is easy to show that it is also in $\hg$.
\end{proof}

\section{Quasi-Jordan matrices}

If $A,B\in\mathfrak{gl}_n(\C)$ are conjugated by an element in $\mathrm{Gl}_n(\C)$, they are even conjugated by an element  in  $\mathrm{Sl}_n(\C)$. But, since $\K$ does not contain all the $n$th roots of its elements,  this fact can not be extended to $\mathrm{gl}_n(\K)$. 

It is known that $s=\sum_{i=-k}^{\infty}a_i t^i\in\K$ has an $n$th root in $\K$ if and only if the \textit{order} of $s$, given by $\mathcal{O}(s):=\mathrm{min}\{m  | \, a_m\neq0\}$,  is a multiple of $n$. Moreover, if $p,q\in\K$ are such that $\mathcal{O} (p)=m$ and $\mathcal{O}(q)=n$, for $m,n\in\mathbb{Z}$, then $\mathcal{O}(pq)=m+n$ and $\mathcal{O}(pq^{-1})=m-n$.

If $\overline{p}=(p_1,...,p_{n-1})\in\K^{n-1}$, a matrix of form 
\begin{equation}\label{jordan}
    J_{\overline{p},n}=\left(
  \begin{array}{ccccc}
    0 & p_1 & 0 & ... & 0 \\
   0  & 0 & p_2 & ... & 0 \\
   \vdots  &  &  & \ddots &  \\
     &  &  & & p_{n-1}   \\
    0 &  & ... &  & 0 \\
  \end{array}
\right)\in\mathrm{gl}_n(\K)
\end{equation} 
is called \emph{quasi-Jordan block}. If $X=D(J_{\overline{p}_{1,i_1}},...,J_{\overline{p}_{d,i_d}} )$ is a block diagonal matrix  with quasi-Jordan blocks $J_{\overline{p}_{k,i_k}}$, we say that $X$ is a \emph{ quasi-Jordan matrix}. Without loss of generality, we establish that $i_1\geq i_2\geq...\geq i_d$. Define  $\mathfrak{m}(J_{\overline{p},n})$, the \textit{multiplicity of} $J_{\overline{p},n}$, and $\mathfrak{n}_X$, the \textit{order of the multiplicities of} $X$, by
\begin{equation}\label{multjordan'}
\mathfrak{m}(J_{\overline{p},n})=p_1^{n-1}p_2^{n-2}...p_{n-1}  \qquad \qquad
\mathfrak{n}_X=\mathcal O\left(\prod_{k=1}^d \m(J_{\overline{q}_{k,j_k}})\right).
\end{equation}

If $X\in\slnk$ is nilpotent, then $X$ is $\Slk$-conjugated to a  quasi-Jordan matrix. To show that, we begin from the existence of a $T\in\mathrm{GL}_n(\K)$ such that $TXT^{-1}=J=\mathrm{diag}(J_{d_1},\cdots,J_{d_m})$ is a Jordan matrix.  Let $q=\det(T)$ and $\mathcal{O}(q)=k$. There are two situations: first, when $k$ is a multiple of $n$, so $q$ has a $n$th root in $\K$. Then $T'=\frac{T}{q^{\frac{1}{n}}}\in\Slk$ and
satisfies that $T'XT'^{-1}=J$. Second, when $k\equiv l(\mathrm{mod \ } n)$ with  $0<l<n$, we have that $\Or(qt^{-l})=k-l$. Then, $qt^{-l}$ has an $n$th root in $\K$. Consider $S=\mathrm{diag}(t^{-l},1,\cdots,1)$ and
$T'=\frac{T}{(qt^{-l})^{\frac{1}{n}}}$. Hence, $T'XT'^{-1}=J$, $\det(T)=t^{-l}$ and
$$(ST')X(ST')^{-1}=SJS^{-1}=\left(\begin{array}{cccccccc}
                                      0 & t^{-l} & 0 &\cdots  & 0 &  &  & \\
                                      0 & 0 & 1 &\cdots  & 0 &  &  &  \\
                                       \vdots  & \vdots &  & \ddots &    \vdots &  &  \\
                                       0  & 0 & 0 & & 1  &  &  \\
                                      0 & 0 & 0 & \cdots& 0  &  &  &  \\
                                       &  &  &  &  & J_{d_2} & &  \\
                                       &  &  &  & &  & \ddots &  \\
                                       &  &  & & &  &  & J_{d_m} \\
                                    \end{array}
                                  \right)$$
is a quasi-Jordan matrix such that $\det(ST')=1$.

Moreover, if $X\in\nk$ is conjugated by $T\in\mathrm{GL}_n(\K)$ to a quasi-Jordan matrix
$D(J_{\overline{p}_{1,d_1}},...,J_{\overline{p}_{m,d_m}} )$ for some $\overline{p}_{i,d_i}=(p_{i,1},...,p_{i,d_i-1})\in\K^{d_i-1}$, then we can explicitly give  the expression of $T$ in terms of $X$.  Let $T\in\Slk$ such that $TXT^{-1}=D(J_{\overline{p}_{1,d_1}},...,J_{\overline{p}_{m,d_m}} )$. Rewrite  $T$ as 
\begin{center} $\left(
     \begin{array}{c}
       T_1 \\
       \vdots \\
       T_i\\
       \vdots \\
       T_m \\
     \end{array}
   \right)$  whith $T_i=\left(
     \begin{array}{c}
       \overline t_{i,1} \\
       \overline t_{i,2} \\
       \vdots\\
       \overline t_{i,d_i -1} \\
       \overline t_{i,d_i} \\
     \end{array}
   \right)$ and $\overline t_{i,j}\in \K^n$
\end{center}
Then, we have that $T_i X=(O_1 \ J_{\overline {p}_{i,d_i}} \ O_2)T$, where $O_1$ and $O_2$ are null matrices of size  $d_i\times(d_1+...+d_{i-1})$ and $d_i\times d_{i+1}+...+d_m$ respectively. From this equality, it follows that 
\begin{align*}
\overline{t}_{i,1} X&=p_{i,1}\overline{t}_{i,2}  &\Rightarrow  \qquad &\overline{t}_{i,2}=p_{i,1}^{-1}\overline{t}_{i,1} X  \\
\overline{t}_{i,2} X&=p_{i,2}\overline{t}_{i,3} &\Rightarrow  \qquad &\overline{t}_{i,3}=p_{i,2}^{-1}\overline{t}_{i,2}X=p_{i,2}^{-1}p_{i,1}^{-1}\overline{t}_{i,1} X^2\\
&\vdots &  \qquad &\vdots  \\
\overline{t}_{i,d_i-1}X&=p_{i,d_i-1}\overline{t}_{i,d_i} &\Rightarrow \qquad &\overline{t}_{i,d_i}=p_{i,d_i-1}^{-1}\overline{t}_{i,d_i-1}X=\cdots=p_{i,d_i-1}^{-1}...p_{i,1}^{-1}\overline t_{i,1}X^{n-1}
\end{align*}

Hence, if we consider $\overline t_{i,1}=\overline t_{i}$ for each $i=1,...,m$, we obtain that 
\begin{equation}\label{Tform}
   T_i=\left(\begin{array}{c}
                                   \overline t_i \\
                                   p_{i,1}^{-1}\overline t_i X \\
                                   \vdots \\
                                   p_{i,1}^{-1}...p_{i,d_{i-2}}^{-1}\overline t_i X^{d_i-2} \\
                                   p_{i,1}^{-1}...p_{i,d_{i-1}}^{-1}\overline t_i X^{d_i-1} \\
                                 \end{array}
                               \right).
                               \end{equation}

\begin{lemma}\label{conjbloq} Let $X=J(p_1,...,p_{n-1})$ and $Y=J(q_1,...,q_{n-1})$ be in $\mathrm{sl}_n(\K)$. Then, $X,Y$ are $\mathrm{Sl}_n(\K)$-conjugated if and only if there exists $r\in\K$ such that
\begin{equation}\label{rfor}
r^n=\mathfrak{m}(X)\mathfrak{m}(Y)^{-1}
\end{equation}
\end{lemma}
\begin{proof} Let $T\in\mathrm{Sl}_n(\K)$ be such that $TXT^{-1}=Y$, then by (\ref{Tform}), on deduces that
\begin{center}
$T=\left(
   \begin{array}{c}
     \overline{t}\\
                              q_1^{-1} \overline{t}X \\
                                \vdots \\
                                q_1^{-1}...q_{n-2}^{-1}\overline{t}X^{n-2} \\
                               q_1^{-1}...q_{n-2}^{-1}q_{n-1}^{-1} \overline{t}X^{n-1} \\
   \end{array}
 \right)$ 
\end{center}
where $\overline{t}=(t_1,...,t_n)\in\K_{1,n}$ for some $t_1,...,t_n\in\K$. Si $T=(t_{ij})$. Hence,
\begin{center}
$t_{ii}=t_1\prod_{j=1}^{i-1}p_j q_j^{-1}$, $1\leq i\leq n$
\end{center}
Since  $\det  T=1$ and $T$ is upper triangular, 
$$\begin{aligned}
\det  T&=\prod t_{ii}\\
&=t_1^np_1^{n-1}p_2^{n-2}...p_{n-1}q_1^{-(n-1)}q_2^{-(n-2)}...q_{n-1}^{-1}\\
&=t_1^n\mathfrak{m}(X)\mathfrak{m}(Y)^{-1}
\end{aligned}$$
Then, considering $r=t_1^{-1}$ we obtain the expected result.

Reciprocally, considering $a_1=1$ and $a_{i+1}=\prod_{k=1}^ip_k q_k^{-1}$ for $i=1,...,n-1$, then the diagonal matrix  $T=\mathrm{diag}(a_1,\cdots,a_n)$ satisfies $TXT^{-1}=Y$. Moreover, (\ref{rfor}) implies that $\mathrm{det \ }T=1$.
\end{proof}

\begin{cor} \label{obsconjbloq} Let $X=J_{(p_1,...,p_{n-1})}$ and $Y=J_{(q_1,...,q_{n-1})}$ be in $\mathrm{sl}_n(\K)$. Then $X,Y$ are $\mathrm{Sl}_n(\K)$-conjugated if and only if
$\Or(\mathfrak{m}(X))-\Or(\mathfrak{m}(Y))$ is a multiple of $n$.
\end{cor}

Now, given $X=D(J_{\overline{p}_{1,i_1}},...,J_{\overline{p}_{d,i_d}} )$ and  $Y=D(J_{\overline{q}_{1,j_1}},...,J_{\overline{q}_{l,j_l}} )$,
$\mathrm{Sl}_n(\K)$-conjugated quasi-Jordan matrices, we obtain that
\begin{enumerate}
 \item $d=l$; that is, they have the same number of blocks, 

 \item $i_k=j_k$ for every $k=1,...,d$; that is, blocks have the same size; in other words, 
 $X,Y$ have the same associated partition.
\end{enumerate}
In fact, $X$ and $Y$ are $\mathrm{Gl}_n(\K)$-conjugated to Jordan matrices 
with blocks of size $i_1,...,i_d$ and
$j_1,...,j_l$ respectively. On the other hand, since $X$ and $Y$ are
$\mathrm{Sl}_n(\K)$-conjugated, they also are
$Gl_n(\K)$-conjugated. So, $X$ and $Y$ are $\mathrm{Gl}_n(\K)$-conjugated  to the same Jordan matrix; which means that
 $d=l$ and $i_k=j_k$ for all $k=1,...,d$.

\begin{theorem}\label{Teo''} Let $X=D(J_{\overline{p}_{1,i_1}},...,J_{\overline{p}_{d,i_d}} )$ and  $Y=D(J_{\overline{q}_{1,j_1}},...,J_{\overline{q}_{l,j_l}} )$ be quasi-Jordan matrices. If $X$ and $Y$ satisfy

 1. $d=l$; that is, they have the same number of blocks, 

 2. $i_k=j_k$ for every $k=1,...,d$; that is, blocks have the same size, and

 3. $\mathcal{O}\left(\prod_{k=1}^d\frac{\m(J_{\overline{q}_{k,j_k}})}{\m(J_{\overline{p}_{k,i_k}})}\right)=\mathfrak n_Y-\mathfrak n_X$, is a multiple of $i_d$, 
then, they are $\mathrm{Sl}_n(\K)$-conjugated.
\end{theorem}
\begin{proof} Given (1) and (2), the normal Jordan form ensures that $X$ e $Y$ are $\mathrm{GL}_n(\K)$-conjugated. Consider $T\in\mathrm{GL}_n(\K)$ as in (\ref{Tform}). Also assume, without loss of generality, that for each vector $\overline{t}_k$ the entry $i_1+\cdots+i_k+1$ is $t_k$  and zero the other ones.
Therefore, $T$ results to be a diagonal block matrix $T=D(T_1,\cdots,T_d)$,  where
$$T_k=\left(
        \begin{array}{ccccc}
          t_k & 0 & 0 & \cdots & 0 \\
          0 & \frac{t_k p_{k,1}}{q_{k,1}} & 0 & \cdots & 0 \\
          0 & 0 & \frac{t_k p_{k,1}p_{k,2}}{q_{k,1}q_{k,2}} &  &  \\
          \vdots &  &  & \ddots &  \\
          0 &  &  &  & \frac{\prod_{j=1}^{i_k-1}p_{k,j}}{\prod_{j=1}^{i_k-1}q_{k,j}} \\
        \end{array}
      \right)$$
for all $k=1,\cdots,d$.  It is clear that such a $T$ satisfy $TXT^{-1}=Y$ and
 $$\det \ T=\prod_{k=1}^d
      t_k^{i_k}\m(J_{\overline{p}_{k,i_k}})\m(J_{\overline{q}_{k,i_k}})^{-1}.$$
 So, if we choose in particular $t_1=\cdots=t_{d-1}=1$ and $t_d$ equal to the $i_d$th root of $\prod_{k=1}^d\frac{\m(J_{\overline{q}_{k,i_k}})}{\m(J_{\overline{p}_{k,i_k}})}$, whose existence is given by (3), then $\det \ T=1$. Therefore, $T\in\Slk$.
\end{proof}

\section{Nilpotent orbits in $\s$}
Consider the group 
$$\LSo:=\C^*\ltimes\LS=\C^*\ltimes Sl_n(\K)$$
and the complition la algebra of $\g=\s$,
$$\hg=\sgor=\slnk\+\C c\+\C d$$

\begin{theorem}\label{nilp4} Let $X\in\s$ be a nilpotent element. Then there exists $g\in\LSo$ such that
$$\mathrm{Ad} \ g(X)=D(J_{\overline{p}_{1,i_1}},...,J_{\overline{p}_{d,i_d}})+\lambda c$$
for some $d\in\mathbb{N}$, $\overline{p}_{k,i_k}\in\K^{i_k}$ and
$\lambda\in\C$.
\end{theorem}
\begin{proof} If $X\in\s$ is nilpotent, by \ref{triang} there exists $g_1\in\LSo$ such that
$\mathrm{Ad} \ g_1 (X)=Y+\lambda_1 c$, where $Y\in\nm$ and
$\lambda_1\in\C$. On the other hand, there exists
$g_2\in\Slk$ such that $\mathrm{Ad}_{\K} \ g_2 (X)=g_2 X
g_2^{-1}=D(J_{\overline{p}_1,d_1},...,J_{\overline{p}_m,d_m})$, for some
 $\overline{p}_{k,i_k}\in\K^{i_k}$ and $d\in\mathbb N$, as was mentioned in the last section. Then,
$$\mathrm{Ad}  \ g_2 g_1 (X)=\mathrm{Ad} \ g_2 (Y+\lambda_1 c
)=D(J_{\overline{p}_1,d_1},...,J_{\overline{p}_m,d_m})+(\lambda_1-\mathrm{res}\langle
g_2^{-1}\frac{dg_2}{dt},Y\rangle _t)c$$ 
Hence, taking $g=g_2 g_1$ and
$\lambda=\lambda_1-\mathrm{res}\langle
g_2^{-1}\frac{dg_2}{dt},Y\rangle _t$, we obtain the expected result.
\end{proof}

\

 Until now, we have proved that every nilpotent element in $\s$ 
 is $\LSo$-conjugated to an element of the form $D+\lambda c$, where $D$ quasi-Jordan block. 
 So, we wonder if  conjugation classes under $\Slk$ and under $\LSo$ coincide.

\begin{theorem} Let $X\in\sgor$ be a nilpotent element. Then, the nilpotent orbits of $X$ in $\sgor$ under the action of $\Slk$ and of $\LSo$ are equal, that is
$$\Ad\Slk( X)=\Ad\LSo( X)$$
\end{theorem}
\begin{proof}
By Theorem \ref{nilp4}, each nilpotent $\LSo$-\'orbit in $\sgor$ has an element of the form $D+\lambda
c\in\sgor$, with $D$ a cuasi-Jordan matrix and $\lambda\in\C$. Hence, we only have to prove that $\Ad \Slk ( D+\lambda
c)=\Ad \LSo ( D+\lambda c)$, or equivalently that for each $z\in\C$,  $\Ad d_z(D+\lambda c)\in\Ad \Slk (D+\lambda c)$. Let $D=D(J_{\overline{p}_{1,i_1}},...,J_{\overline{p}_{d,i_d}})$, where $J_{\overline{p}_{j,i_j}}=J(p_{j 1},...,p_{j,i_j})$ and $(p_{j,1},...,p_{j,i_j})\in\K^{i_j}$. Then,
\begin{equation}
\Ad d_z(D+\lambda c)=D(J(p_{1,1}(zt),...,p_{1,i_1}(zt)),...,J(p_{d,1}(zt),...,p_{d,i_d}(zt)))+\lambda c
\end{equation}
Denote $D_z=D(J(p_{1,1}(zt),...,p_{1,i_1}(zt)),...,J(p_{d,1}(zt),...,p_{d,i_d}(zt)))$. So, $\mathcal{O}(p(t))=\mathcal{O}(p(zt))$ and $\mathcal{O}(p)=-\mathcal{O}(p^{-1})$ for all $p\in\K$ and $z\in\C$. 
Then, for each $j=1,...,d$, we have that
\begin{equation}
\mathcal{O}(p_{j,1}(t)^{i_j-1}...p_{j,i_j}(t))+\mathcal{O}(p_{j,1}(zt)^{-i_j+1}...p_{j,i_j}(zt)^{-1})=0
\end{equation}
Therefore, by Theorem \ref{Teo''}, the matrices $D$ and $D_z$ are $\Slk$-conjugated.
\end{proof}

\subsection{Level of nilpotent orbits}

The following three lemmas are technical results we need to prove that conjugated nilpotent elements, whose components in $\slnk$ are quasi-Jordan matrices, have the same component in $\C c$. That allows us to define the level of a nilpotent orbit.

\begin{lemma} \label{entnulasT} Let $X=D(J_{\overline{p}_{1,d_1}},...,J_{\overline{p}_{m,d_m}} )$, for some
$\overline{p}_{i,d_i}=(p_{i,1},...,p_{i,d_i-1})\in\K^{d_i-1}$ y, $T=(t_{ij})\in\mathrm{SL}_n(\K)$ such that $TXT^{-1}$  is a quasi-Jordan matrix. If
\begin{center}
$i=r+d_1+\cdots+d_{\alpha}, \mbox{ with } r\leq d_{\alpha}$

$j=s+d_1+\cdots+d_{\beta}, \mbox{ with } s\leq d_{\beta}$
\end{center}
and $1\leq s<r\leq d_1$, then $t_{ij}=0$
\end{lemma}
\begin{proof}
Let $i=r+d_1+\cdots+ d_{\alpha}, j=s+d_1+\cdots+d_{\beta}$, with $r\leq d_{\alpha}, s\leq d_{\beta}$ and $1\leq s<r\leq d_1$. By Lemma \ref{conjbloq}, there exits $\overline{t}_{\alpha+1}\in\K^n$ such that
$$t_{ij}=\overline{t}_{\alpha+1}\cdot\mbox{col}_j X^{r-1}$$ 
It is easy to show that the columns  $\lambda+d_1+\cdots+d_{\beta}$ of
$X^{r-1}$, for $1\leq\lambda\leq r-1$, are null. So, since $s<r$, the column $j$ of $X^{r-1}$ is null. Hence, $t_{ij}=0$.
\end{proof}

\begin{lemma}\label{entnulasminors} Let $X=D(J_{\overline{p}_{1,d_1}},...,J_{\overline{p}_{m,d_m}} )$, with $\overline{p}_{i,d_i}=(p_{i,1},...,p_{i,d_i-1})\in\K^{d_i-1}$, and $T=(t_{ij})\in\mathrm{SL}_n(\K)$ such that $TXT^{-1}$ is a quasi-Jordan matrix. Let $M_{ij}$ be the determinant of the matrix obtained from $T$ eliminating the row $i$ and the column $j$. If
\begin{center}
$i=r+d_1+\cdots+d_{\alpha}, \mbox{ with } r\leq d_{\alpha}$

$j=s+d_1+\cdots+d_{\beta}, \mbox{ with } s\leq d_{\beta}$
\end{center}
and $1\leq r<s\leq d_1$, then $M_{ij}=0$.
\end{lemma}

\begin{proof}
Denote by
\begin{equation}\label{tren}
j_1=1; \ \ j_l=1+d_1+\cdots+d_{l-1} \mbox{ if } 2\leq l\leq m-1
\end{equation}
and consider the following sets for $0\leq u<d_1$
\begin{equation}
    I_u=\{j_1+u,\cdots,j_{r_u}+u\} \mbox{, donde } d_{r_u}\geq u+1
    \mbox{ and } d_{r_u+1}<u+1.
\end{equation}
Let $i=r+d_1+\cdots+d_{\alpha}$ and $j=s+d_1+\cdots+d_{\beta}$, with
$r<d_{\alpha}$, $s<d_{\beta}$ and $1\leq r<s\leq d_1$. 
Consider two cases:
\begin{itemize}
\item If $i\neq j_l$ for all $1\leq l\leq m-1$, 
\begin{equation*}
    M_{ij}=\sum_{k=1}^m\pm
    t_{j_k,j_l}M_{(i,j_k / j,j_m)}
\end{equation*}
Fix $1\leq k\leq m$ and take
$S=\{1,\cdots,n\}\smallsetminus\{j,j_m\}$,
$S'=\{1,\cdots,n\}\smallsetminus\{i,j_k\}$ and 
$\Omega=\{\sigma:S\rightarrow S' \mbox{ inyectiva}\}$.

Consider $\sigma\in\Omega$  satisfying that $t_{\sigma_a,a}\neq 0$ for all
$a\in S$, then by Lemma \ref{entnulasT}, we have that:

- If $a\in I_0\smallsetminus\{j_m\}$, then $\sigma_a\in
I_0\smallsetminus\{j_k\}$.

- If $1<\lambda<r$ y $a\in I_{\lambda-1}$, then $\sigma_a\in
I_{\lambda-1}$.

- If $a\in I_{r-1}$, then $\sigma_a\in I_{r-1}\smallsetminus\{i\}$.

Hence, there is $a\in I_{r-1}$ such that
$g_{\sigma_a,a}=0$. Therefore, 
$\prod_{a\in S}t_{\sigma_a,a}=0$ for all $\sigma\in\Omega$; then $M_{ij}=0$.

\item If $i=j_l$ for some $1\leq l\leq m-1$, 
\begin{equation*}
    M_{ij}=\sum_{\substack{{k=1} \\ k\neq l}}^m\pm
    t_{j_k,j_l}M_{(i,j_k / j,j_m)}
\end{equation*}
Consider $S, S'$ and $\Omega$ as in the previous case. We have that $t_{\sigma_a,a}\neq 0$ if and only if   $\sigma_a\in I_0\smallsetminus\{j_l,j_k\}$ for all $a\in
I_0\smallsetminus\{j_m\}$. Therefore, 
$\prod_{a\in S}t_{\sigma_a,a}=0$ for all $\sigma\in\Omega$, which means that $M_{ij}=0$.
\end{itemize}
\end{proof}

\begin{lemma}\label{hoja} Let $X=D(J_{\overline{p}_{1,d_1}},...,J_{\overline{p}_{m,d_m}} )$, with $\overline{p}_{i,d_i}=(p_{i,1},...,p_{i,d_i-1})\in\K^{d_i-1}$ and
$T=(t_{ij})\in\mathrm{SL}_n(\K)$ such that $TXT^{-1}$  is a quasi-Jordan matrix. Let $M_{ij}$ be the determinant of the matrix obtained from  $T$ eliminating the row $i$ and the column $j$. If $t_{ij}\neq 0$, then $M_{i,j+1}=0$.
\end{lemma}
\begin{proof} Let $t_{ij}\neq 0$, by Lemma \ref{entnulasT}, we know that $i=r+d_1+\cdots+d_{\alpha}$, $j=s+d_1+\cdots+d_{\beta}$, with $r\leq d_{\alpha}$, $s\leq d_{\beta}$ and $1\leq r\leq s\leq d_1$. Since $r<s+1$, by Lemma \ref{entnulasminors},  $M_{i,j+1}=0$.
\end{proof}

\

In the following lemma, the integers $j_i$ are as in (\ref{tren}).

\begin{lemma}\label{carro} Let $X=D(J_{\overline{p}_{1,d_1}},...,J_{\overline{p}_{m,d_m}} )$, with $\overline{p}_{i,d_i}=(p_{i,1},...,p_{i,d_i-1})\in\K^{d_i-1}$, and
$T=(t_{ij})\in\mathrm{SL}_n(\K)$ such that $TXT^{-1}$  is a
quasi-Jordan matrix. Then, for all
$i\in\{1,\cdots,n\}\smallsetminus\{j_2-1,\cdots,j_d-1\}$, the entry
$(i+1,i)$ of $T^{-1}\frac{dT}{dt}$ is zero.
\end{lemma}
\begin{proof} The entry $(i+1,i)$ of $T^{-1}\frac{dT}{dt}$ is
\begin{equation*} 
    \begin{split}
       \left(T^{-1}\frac{dT}{dt}\right)_{i+1,i}  & =\mathrm{fil}_{i+1}T^{-1}\cdot\mathrm{col}_i\frac{dT}{dt} \\
         & = \mathrm{col}_{i+1}\mathrm{cof}(T)\cdot\mathrm{col}_i
         \frac{dT}{dt}\\
         & = \sum_{j=1}^n
         \mathrm{cof}(T)_{j,i+1}\cdot\frac{dT}{dt}_{j,i},
     \end{split}
\end{equation*}
where $\mathrm{cof}(T)$ denotes the matrix of cofactors of $T$. So, by the Lemmas \ref{entnulasT} and \ref{hoja}, 
each term of the above sum are zero.
\end{proof}

\begin{theorem}\label{lambda} Let $\overline D=D+\lambda c$ and
$\overline D '=D'+\lambda' c$, where $D$ and $D'$ are
quasi-Jordan matrices. If $\mathfrak{O}_{\overline
D}=\mathfrak{O}_{\overline D'}$, then $\lambda=\lambda'$. Moreover, if $g Dg^{-1}=D'$ for some $g\in\Slk$, then $\Ad g(D)=D'$.
\end{theorem}
\begin{proof} Let $g\in\LGo$ be such that $\Ad  g(\overline D)=\overline D'$. hence,
\begin{equation}
\Ad g(\overline D)=gDg^{-1}+(\lambda-\mathrm{res}\langle
g^{-1}\frac{dg}{dt},D\rangle_t)c
\end{equation}
This means that $D'=gDg^{-1}$ and
$\lambda'=\lambda-\mathrm{res}\langle
g^{-1}\frac{dg}{dt},D\rangle_t$.

 If we express $D$ and
$g^{-1}\frac{dg}{dt}$ in $\sgor$ as a sum of root vectors with
respect to $\Pi$, it results that the root vectors occurring in $D$ are the ones corresponding to the roots of the form  $\alpha_i+n\delta$, where
$\alpha_i$ is a simple root of $\dot\g$, for
$i\in\{1,\cdots,n\}\smallsetminus\{j_2-1,\cdots,j _d-1\}$ and $n\in\mathbb
Z$;  while the root vectors corresponding to these roots do not occur in 
$g^{-1}\frac{dg}{dt}\in\hg$ by Lemma \ref{carro}. Therefore,
 $\langle
g^{-1}\frac{dg}{dt},D\rangle_t=0$. Hence, $\lambda=\lambda'$
\end{proof}

\

If $X\in\hg$ is nilpotent, we have shown that $X$ is
$\LSo$-conjugated to an element of the form  $D+\lambda c$ where $D$ is
a quasi-Jordan matrix and $\lambda\in\C$. Theorem \ref{lambda} leads to the following definition.

\begin{defn} Let $\mathfrak O_X$ be a nilpotent orbit in $\hg$. If $D+\lambda c\in\mathfrak O_X$, with $D$ a quasi-Jordan matrix and $\lambda\in\C$, then   $\lambda$ is the \emph{level of} $\mathfrak{O}_X$.
\end{defn}

\section{Parameterization}

In this section we construct the parameterization of nilpotent orbits of $\sgor$.

\begin{theorem}\label{pollo} Every nilpotent element in $\sgor$ is conjugated by $\LSo$ to an element of the form  $D_{\sigma,k}+\lambda c$, where $\sigma=[i_1,\cdots,i_d]$ is a partition of $n$, $0\leq k<i_d$, $\lambda\in\C$ and
\begin{equation}\label{gato}
D_{\sigma,k}=\left(
  \begin{array}{cccccccc}
    J_{i_1} &  &  &  &  &  &  & \\
     & \ddots & &  &  &  &  &  \\
     &  & J_{i_{d}-1} &  &  &  &  &  \\
     &  &  & 0 & 1 & 0 & \cdots & 0 \\
     &  &  &  &  & \ddots &  &  \\
     &  &  & \vdots &  &  & 1 & 0 \\
     &  &  &  &  &  & 0 & t^k \\
     &  &  & 0 &  &  & 0 & 0 \\
  \end{array}
\right)
\end{equation}
\end{theorem}

\begin{proof} Let $X\in\sgor$ be a nilpotent element, then 
by Theorem \ref{nilp4}, there exits $g_1\in\LSo$ such that $\Ad g_1 (X)=D+\lambda c$,
where $D=D(J_{\overline{p}_{1,i_1}},...,J_{\overline{p}_{d,i_d}} )$
is a quasi-Jordan matrix. Let
$\mathfrak n_D$ be the order of the 
multiplicities of $D$ as in  (\ref{multjordan'}). Take $\sigma=[i_1,\cdots,i_d]$ and  $k$ so that
 $0\leq k<i_d$ and $k\equiv \mathfrak n_D(\mathrm{mod \ }i_d)$.
 Hence, by Theorem \ref{Teo''},
there exists $g_2\in\Slk$ such that $g_2 Dg_2^{-1}=D_{\sigma,k}$. So, if $g=g_2 g_1$, we have that
$\Ad g (X)=D_{\sigma,k} +\lambda c$.
\end{proof}

\begin{obs} The last theorem implies that given a partition  
$\sigma=[i_1,\cdots,i_d]$ of $n$ and $\lambda\in\C$, there are at most 
$i_d$ nilpotent orbits in $\sgor$ with partition $\sigma$ of level
$\lambda$.
\end{obs}

\begin{prop}\label{k=k'} Let $\sigma=[i_1^{k_1},i_2^{k_2},\cdots,i_s^{k_s}]$ be a
partition of $n$ such that $i_1>i_2>\cdots>i_s$, and $0\leq k',
k<i_s$. Then $D_{\sigma,k}$ and $D_{\sigma,k'}$ are $\Slk$-conjugated if and only if
$k=k'$.

\end{prop}
\begin{proof} Suposse there exists
$T\in\Slk$ as in (\ref{Tform}), conjugated to $D_k=D_{\sigma,k}$ and to $D_{k'}=D_{\sigma,k'}$.

1. Denote by
\begin{equation}
l_1=0; \ \ l_m=\sum_{j=1}^{m-1}k_j i_j, \mbox{ for } 2\leq m\leq s
\end{equation}

Let $M$ be the matrix obtained by taking from  $T$ the rows and columns 
 $$L_{t,r}=1+l_t+(r-1)i_t, \mbox{ con } 1\leq t\leq s \mbox{ y } 1\leq r\leq
 k_t$$
 and $M_j$ the matrix obtained by taking the first 
 $k_1+\cdots+k_j$ rows and columns from $M$, for $1\leq j\leq s$. Particularly, $M=M_s$.
Set $$T_{r,p}^{u,v}=t_{L_{r,u},L_{p,v}}$$

 2.  For each $1\leq r\leq s$
 \begin{equation}\label{caza}
 T_{r,p}^{u,v}=0, \mbox{ if } 1\leq p\leq r-1, \  1\leq u\leq k_r  \mbox{ and } 1\leq v\leq k_p
 \end{equation}

Fix $1\leq r\leq s$ and take $1\leq u\leq k_r$. By
Proposition \ref{conjbloq}, 
$$\mathrm{fil}_{L_{r,u}+(i_r-1)}(TD_k)=\mathrm{fil}_{L_{r,u}}(T)D_k^{i_r}$$
So, as $i_p<i_r$ if $1\leq p<r$, then
$L_{p,v},L_{p,v}+i_r$ of $D_k^{i_r}$ is 1 if $1\leq v\leq k_p$. Therefore,
$$(TD_k)_{L_{r,u},L_{p,v}+i_r}=t_{L_{r,u},L_{p,v}}=T_{r,p}^{u,v}.$$
On the other hand,
$\mathrm{fil}_{L_{r,u}+(i_r-1)}(D_{k'}T)=\mathrm{fil}_{L_{r,u}+(i_r-1)}(D_{k'})T$
is null. Since $TD_k=D_{k'}T$, the statement (\ref{caza}) occurs.

3. According to the notation in (1), by (2) 
$$M_j=\left(
      \begin{array}{cccc}
        M_{j-1} & [T_{r,j}^{u,v}]_{1\leq r\leq j-1}  \\
        0 & [T_{j,j}^{u,v}]   \\
      \end{array}
    \right)$$
for each $1\leq j\leq s$. Such that $\det M_j=\det
M_{j-1}\det[T_{j,j}^{u,v}]$

 4. We have that $\det T=t^{k-k'}(\det M_1)^{i_1-i_2}(\det M_2)^{i_2-i_3}\cdots(\det
 M_{s-1})^{i_{s-1}-i_s}(\det M_s)^{i_s}$.

5. Let $P_i=\det M_i$. Then, by (3) and (4),  $\det
T=t^{k-k'}\prod_{j=1}^s P_j^{m_j}$,
 where $m_j\geq i_s$ for all $1\leq j\leq s$.
If we consider $1\leq k<k'\leq i_s$, then $0<k'-k<i_s$.
Hence, there are no elements $P_i$ in $\K$ for which
 $t^{k'-k}=\prod_{j=1}^s P_j^{m_j}$ if $m_j\geq
i_s$. Therefore, it is not possible to obtain $T\in\Slk$ such that $TD_k
T^{-1}=D_{k'}$.
\end{proof}

The parameterization of nilpotent orbits in $\sgor$ is given by the following result.
\begin{theorem}\label{clas}  There is a biyective correspondence between nilpotent orbits in
 $\sgor$ and the set
 $\{(\sigma,j):\sigma=[i_1,...,i_d]\in\mathcal{P}(n) \
\mathrm{ y } \ 0\leq j<i_d\}\times\C$.
\end{theorem}
\begin{proof} Denote  the set of nilpotent orbits in $\sgor$ by $\mathfrak O(\sgor)$ .
Consider the application
\begin{equation}
\Psi:\mathfrak
O(\sgor)\rightarrow\{(\sigma,j):\sigma=[i_1,...,i_d]\in\mathcal{P}(n) \
\mathrm{ and } \ 0\leq j<i_d\}
\end{equation}
 defined as follows: Let $X\in\sgor$ be a 
nilpotent element, then by Corollary \ref{pollo}, there exist
$g\in\LSo$, a partition $\sigma=[i_1,...,i_d]$ of $n$ and 
$0\leq k<i_d$, such that $\mathrm{Ad} g X=D_{\sigma,k}+\lambda c$.
Hence,
\begin{equation}\label{def.}
\Psi(\mathfrak{O}_X)=([i_1,...,i_d],k,\lambda)
\end{equation}

Observe that:

 1. $\Psi$ is well defined: Let $X,Y\in\sgor$ be $\LSo$-conjugated nilpotent elements. Suposse that there exist $g_1,g_2\in\LSo$,
 a partition $\sigma_1=[i_1,\cdots,i_d], \sigma_2=[j_1,\cdots,j_r]$ of
 $n$, $0\leq k_1<i_d$, $0\leq k_2<j_r$ and $\lambda_1,\lambda_2\in\C$
 such that
 $$\Ad g_1( X)=D_{\sigma_1,k_1}+\lambda_1 c, \ \Ad g_2
 (Y)=D_{\sigma_2,k_2}+\lambda_2 c.$$
Since $D_{\sigma_1,k_1}+\lambda_1 c$ and
$D_{\sigma_2,k_2}+\lambda_2 c$ are
$\LSo$-conjugated, then by Corollary \ref{obsconjbloq}, we have that
$\sigma_1=\sigma_2$; by the Proposition \ref{k=k'}, $k_1=k_2$ and
by Theorem \ref{lambda}, $\lambda_1=\lambda_2$. Therefore, $\Psi$ is well defined.

 2. $\Psi$ is inyective: if $X,Y$ are nilpotent elements in $\s$,
 and
$$
\Psi(\mathfrak{O}_X)=\Psi(\mathfrak{O}_Y)=(\sigma=[i_1,\cdots,i_d],k,\lambda)
$$
then there exit $g_1,g_2\in\LSo$ such that $\Ad g_1
 X=D_{\sigma,k}+\lambda c=\Ad g_2 Y$. Hence, $X$ and $Y$ are
 $\LSo$-conjugated, so $\mathfrak{O}_X=\mathfrak{O}_Y$.

3. $\Psi$ is suryective: Let $\sigma=[i_1,\cdots,i_d]$ be a partition of $n$, $0\leq k<i_d$ and $\lambda\in\C$. 
Then
$\Psi(\mathfrak{O}_{D_{\sigma,k}+\lambda c})=(\sigma,k,\lambda)$.
\end{proof}

\begin{obs} If $\lambda\in\C$, then the number of nipotent orbits in 
 $\sgor$ of level $\lambda$ is finite.
\end{obs}

\begin{exmp} Let $\g=\mathrm{sl}_4^{(1)}(\C)$. The partition of $4$ are $[1,1,1,1], [1,1,2],[1,3],[2,2],[4]$, then $\mathfrak{sl}_4(\C)$ has $5$ nilpotent orbits. 
The nilpotent orbits of level 0 in $\hg$ are the following.

\begin{center}
\begin{tabular}{|c|c|c|}
  \hline
  Partici\'on &      k        & $X_{\mathfrak d}$ \\ \hline
   $[1,1,1,1]$& 0 & $\left(
                                  \begin{array}{cccc}
                                    0 & 0 & 0 & 0 \\
                                    0 & 0 & 0 & 0 \\
                                    0 & 0 &0 & 0 \\
                                    0 & 0 & 0 & 0 \\
                                  \end{array}
                                \right)$ \\ \hline
   $[2,1,1]$  & 0 & $\left(
                                  \begin{array}{cccc}
                                    0 & 1 & 0 & 0 \\
                                    0 & 0 & 0 & 0 \\
                                    0 & 0 &0 & 0 \\
                                    0 & 0 & 0 & 0 \\
                                  \end{array}
                                \right)$  \\  \hline
      $[2,2]$        & 0  & $\left(
                                  \begin{array}{cccc}
                                    0 & 1 & 0 & 0 \\
                                    0 & 0 & 0 & 0 \\
                                    0 & 0 &0 & 1 \\
                                    0 & 0 & 0 & 0 \\
                                  \end{array}
                                \right)$ \\  \hline
   $[2,2]$    & 1     & $\left(
                                  \begin{array}{cccc}
                                    0 & 1 & 0 & 0 \\
                                    0 & 0 & 0& 0 \\
                                    0 & 0 &0 & t \\
                                    0 & 0 & 0 & 0 \\
                                  \end{array}
                                \right)$ \\  \hline
         $[3,1]$     &  0     & $\left(
                                  \begin{array}{cccc}
                                    0 & 1 & 0 & 0 \\
                                    0 & 0 & 1 & 0 \\
                                    0 & 0 &0 & 0 \\
                                    0 & 0 & 0 & 0 \\
                                  \end{array}
                                \right)$ \\  \hline
           $[4]$   &   0    & $\left(
                                  \begin{array}{cccc}
                                    0 & 1 & 0 & 0 \\
                                    0 & 0 & 1 & 0 \\
                                    0 & 0 &0 & 1 \\
                                    0 & 0 & 0 & 0 \\
                                  \end{array}
                                \right)$     \\  \hline
   $[4]$    & 1     & $\left(
                                  \begin{array}{cccc}
                                    0 & 1 & 0 & 0 \\
                                    0 & 0 & 1 & 0 \\
                                    0 & 0 &0 & t \\
                                    0 & 0 & 0 & 0 \\
                                  \end{array}
                                \right)$  \\ \hline
   $[4]$ & \ 2 \ &  $\left(
                                  \begin{array}{cccc}
                                    0 & 1 & 0 & 0 \\
                                    0 & 0 & 1 & 0 \\
                                    0 & 0 &0 & t^2 \\
                                    0 & 0 & 0 & 0 \\
                                  \end{array}
                                \right)$ \\ \hline
   $[4]$ & 3  &  \ $\left(
                                  \begin{array}{cccc}
                                    0 & 1 & 0 & 0 \\
                                    0 & 0 & 1 & 0 \\
                                    0 & 0 &0 & t^3 \\
                                    0 & 0 & 0 & 0 \\
                                  \end{array}
                                \right)$ \ \\ \hline

\end{tabular}
\end{center}
\end{exmp}

\providecommand{\href}[2]{#2}

\end{document}